\documentclass[11pt,reqno]{amsart}
\usepackage{amscd,graphicx,psfrag,amssymb,color}

\newcommand{\p}{\partial}

\newcommand{\Z}{\mathbb Z}

\renewcommand{\phi}{\varphi}

\newcommand{\ad}{\operatorname{ad}}

\newcommand{\lk}{\operatorname{lk}\,}

\newcommand{\Int}{\operatorname{int}}

\newcommand{\I}{\mathcal I}

\newtheorem{theorem}{Theorem}[section]
\newtheorem{lemma}[theorem]{Lemma}

\newtheorem{corollary}[theorem]{Corollary}

\theoremstyle{definition}
\newtheorem*{remark}{Remark}
\newtheorem*{example}{Example}

\author[Eric Harper]{Eric Harper}
\address{CIRGET\newline\indent
Universit\'e du Qu\'ebec a Montr\'eal \newline\indent Case Postale 8888, Succursale Centre-ville \newline\indent Montr\'eal, QC, H3C 3P8 }
\email{\rm{harper@cirget.ca}}

\author[Nikolai Saveliev]{Nikolai Saveliev}
\address{Department of Mathematics\newline\indent
University of Miami \newline\indent PO Box 249085
\newline\indent Coral Gables, FL, 33124}
\email{\rm{saveliev@math.miami.edu}}

\begin{document}

\title{Instanton Floer homology for two-component links}

\maketitle


\section{Introduction}
The instanton Floer homology for integral homology spheres was defined by 
Andreas Floer \cite{floer1}. It has been making a comeback lately, in 
particular, in the work of Lim \cite{lim} and Kronheimer-Mrowka \cite{KM1} 
on the instanton Floer homology of knots. The interest in this 
particular version of the theory, which first appeared in Floer \cite{floer2}, 
is explained by its conjectured relationship with the Seiberg--Witten and 
Heegaard Floer homologies. This is evidenced, for instance, by the fact 
that the Alexander polynomial of the knot can be expressed in terms of its 
instanton Floer homology; see \cite{KM1} and \cite{lim}. 

In this paper we define instanton Floer homology for links of two 
components in an integral homology sphere by a slight variant of the 
construction of Floer \cite{floer2} (which was later developed by Braam and 
Donaldson \cite{BD}). We show that the Euler characteristic of this homology 
theory is twice the linking number between the components of the link. A 
similar result for two-component links in the 3-sphere was announced by Braam 
and Donaldson \cite[Part II, Example 3.13]{BD} with an outline of the proof 
that relied on the Floer exact triangle. Our approach is more direct and its 
main advantage is that it yields the result for links in arbitrary homology 
spheres. 

We provide several examples of calculations of Floer homology groups of
links. We also observe that our Floer homology can be viewed as a reduced 
version of Kronheimer and Mrowka's instanton Floer homology for links.

Our interest in the Floer homology for two-component links was generated by 
the papers \cite{KM1} and \cite{lim} and by our own study of the linking 
number in~\cite{HS}.

We are thankful to the referee for useful remarks which helped us improve 
the paper. 


\section{The instanton Floer homology $\I_*(\Sigma,L)$}
Let $L = \ell_1 \cup \, \ell_2$ be an oriented link of two components in 
an integral homology sphere $\Sigma$. Consider the link exterior $X = 
\Sigma - \Int N(L)$ and furl it up by gluing the boundary components of $X$ 
together via an orientation reversing diffeomorphism $\phi : T^2 \to T^2$.
The resulting closed orientable manifold will be denoted $X_{\phi}$. 

\begin{lemma}\label{L:choice}
The gluing map $\phi$ can be chosen so that $X_{\phi}$ has the integral 
homology of $S^1 \times S^2$.
\end{lemma}

\begin{proof}
An orientation reversing diffeomorphism $\phi$ is determined up to isotopy 
by the homomorphism $\phi_*: \Z^2 \to \Z^2$ it induces on the fundamental
groups of the two boundary components of $X$. Let $\mu_1$, $\lambda_1$ be 
the canonical oriented meridian--longitude pair on one boundary component 
of $X$, and $\mu_2$, $\lambda_2$ on the other. With respect to this choice 
of bases, $\phi_*$ is given by an integral matrix 
\[
\phi_* = 
\left( \begin{array}{cc} a & c \\ b & d \end{array} \right)\quad
\text{with}\quad ad - bc = -1.
\]
Now $H_1 (X_{\phi})$ has generators $t$, $\mu_1$, $\mu_2$ and relations
\[
\begin{array}{lcl} \mu_2 & = & a\mu_1 + b\lambda_1, \\ 
\lambda_2 & = & c\mu_1 + d\lambda_1, \end{array}
\]
where $\lambda_1 = n\mu_2$ and $\lambda_2 = n\mu_1$ with $n = \lk(\ell_1,
\ell_2)$. In particular, $X_{\phi}$ has the integral homology of $S^1 \times 
S^2$ exactly when
\begin{equation}\label{E:eq}
\det \left(\begin{array}{cc} a & bn-1 \\ c-n & dn \end{array} \right) = \;
bn^2 - 2n + c\; =\; \pm 1.
\end{equation}
It is clear that one can always find $\phi$ such that this is the case. 
\end{proof}

\begin{remark}
This construction appeared in Brakes \cite{brakes} and Woodard \cite{wood} 
under the name of ``sewing-up link exteriors'', and was generalized by Hoste 
in \cite{hoste}. We thank Daniel Ruberman for drawing our attention to these 
papers. 
\end{remark}

Given a link $L$ of two components in an integral homology sphere $\Sigma$, 
choose the gluing map $\phi$ so that $H_* (X_{\phi}) = H_* (S^1 \times S^2)$ 
and let
\[
\I_* (\Sigma,L) = I_* (X_{\phi}).
\]

Here, $I_* (X_{\phi})$ is the instanton Floer homology defined in Floer 
\cite{floer2}, see also Braam--Donaldson \cite{BD}, as follows. Let $E$ be 
a $U(2)$--bundle over $X_{\phi}$ such that $c_1(E)$ is an odd element in 
$H^2 (X_{\phi}) = \mathbb Z$. Consider the space of $PU(2)$--connections in 
the adjoint bundle $\ad (E)$ modulo the action of the gauge group consisting 
of automorphisms of $E$ with determinant one. The Floer homology arising 
from the Chern--Simons functional on this space is $I_*(X_{\phi})$. It has a 
relative grading by $\mathbb Z/8$. 

We will refer to $\I_* (\Sigma,L)$ as the \emph{instanton Floer homology} of 
the two-component link $L \subset \Sigma$.

\begin{theorem}\label{linking}
Let $L = \ell_1 \cup \, \ell_2$ be an oriented two-component link in an 
integral homology sphere $\Sigma$. Then $\I_* (\Sigma,L)$ is independent 
of the choice of $\phi$, and its Euler characteristic is given by
\[
\chi(\I_*(\Sigma,L))\; =\; \pm\, 2 \lk(\ell_1 \cup \, \ell_2).
\]  
\end{theorem}

\begin{proof}
The first statement follows from the excision principle of Floer \cite{floer2},
see also \cite[Part II, Proposition 3.5]{BD}. Since $X_{\phi}$ is a homology 
$S^1 \times S^2$, we know that $\chi(I_*(X_{\phi})) = \pm \Delta''(1)$, where 
$\Delta(t)$ is the Alexander polynomial of $X_{\phi}$ normalized so that 
$\Delta(1) = 1$ and $\Delta(t) = \Delta(t^{-1})$ (a direct proof of this 
result can be found in \cite{masa}). To calculate $\Delta(t)$, let 
$\widetilde X_{\phi}$ be the infinite cyclic cover of $X_{\phi}$. Then 
$H_1(\widetilde X_{\phi})$, as a $\Z[t,t^{-1}]$--module, has generators $\mu_1$, 
$\mu_2$ and relations 
\[
\begin{array}{lcl} t\mu_2 & = & a\mu_1 + b\lambda_1, \\ 
t\lambda_2 & = & c\mu_1 + d\lambda_1, \end{array}
\]  
where $\lambda_1 = n \mu_2$ and $\lambda_2 = n \mu_1$. Therefore,
\[
\Delta(t) = \det \left( \begin{array}{cc} a & bn-t \\ 
c-nt & dn \end{array} \right) \\
= -nt^2 + (bn^2 + c)t - n, 
\]
up to a unit in $\Z[t,t^{-1}]$. After taking \eqref{E:eq} into account, we 
obtain 
\[
\Delta(t) =  \pm (-nt + (2n \pm 1) - nt^{-1})
\]
so that
\[
\Delta''(1) = \pm\, 2 n = \pm\,2 \lk(\ell_1,\ell_2).
\]
\end{proof}

\begin{remark}
The requirement that $X_{\phi}$ have homology of $S^1 \times S^2$ was only
needed to make the discussion more elementary, and in general can be omitted. 
Braam and Donaldson \cite[Part II, Proposition 3.5]{BD} show that \emph{any} 
choice of $\phi$ gives an admissible object $X_{\phi}$ in Floer's category, 
and that the properly defined Floer homology of $X_{\phi}$ is independent of 
$\phi$. 
\end{remark}


\section{Non-triviality}
For two-component links $L = \ell_1\cup\ell_2$ with linking number $\lk 
(\ell_1,\ell_2) \neq 0$, the instanton Floer homology $\I_*(\Sigma,L)$ 
must be non-trivial by Theorem~\ref{linking}. In this section, we will 
give a sufficient condition for $\I_*(\Sigma,L)$ to be non-trivial even 
when $\lk(\ell_1,\ell_2) = 0$.

Let $Y$ be a closed, irreducible, orientable 3-manifold, and let $0 \neq 
v \in H^2 (Y;\Z/2)$. In the proof of~\cite[Theorem 3]{KM2}, Kronheimer 
and Mrowka show that the instanton Floer homology of $Y$ constructed from
$PU(2)$ connections in the bundle $\ad (E)$ with $w_2 (\ad (E)) = v$ must
be non-trivial. This implies that $\I_*(\Sigma,L)$ is non-trivial whenever 
$X_{\phi}$ is irreducible. A standard $3$--manifold topology argument can 
be used to show that the irreducibility of $X$ implies that of $X_{\phi}$, 
which leads to the following theorem.

\begin{theorem}  
Let $L = \ell_1 \cup \, \ell_2$ be an oriented two-component link in an 
integral homology sphere $\Sigma$ such that the link exterior is irreducible.
Then $\I_* (\Sigma,L)$ is non-trivial.
\end{theorem}

Since link exteriors are irreducible for non-split links in the 3-sphere, 
we conclude the following: 

\begin{corollary}
For all non-split, two-component links in $S^3$, the Floer homology $\I_* 
(S^3,L)$ is non-trivial.
\end{corollary} 

On the other hand, for any split link $L \subset \Sigma$, we have $\I_*
(\Sigma,L) = 0$ since $w_2$ must evaluate non-trivially on $S^2 \subset 
X_\phi$, but there are no non-trivial flat $SO(3)$ connections on $S^2$ 
due to the fact that $\pi_1 (S^2) = 1$. 
       

\section{A surgery description}
Let $L \subset \Sigma$ be an oriented two-component link in a homology $3$-sphere. 
In what follows we will give a surgery description of a manifold $Y$ such that 
$\I_*(\Sigma,L) = I_*(Y)$. This will allow us to calculate $\I_*(\Sigma,L)$ 
for several examples.

Attach a band from one component of $L$ to the other matching orientations,
and call the resulting knot $k$. Introduce a small circle $\gamma$ going
once around the band with linking number zero. Frame $\gamma$ by zero and 
$k$ by an integer $m = \pm 1$ such that $m$--surgery along $k$ results in a 
homology sphere $\Sigma + m \cdot k$.  Any manifold obtained from $\Sigma$ by 
performing surgery on the framed link $k\,\cup\,\gamma$ will be called $Y$. 
According to \cite[Part II, Proposition 3.5]{BD}, see also \cite{hoste}, the 
manifold $Y$ is diffeomorphic to $X_{\phi'}$ for a choice of gluing map $\phi'$. 
Since $Y$ has the integral homology of $S^1\times S^2$, the map $\phi'$ must 
be as in Lemma \ref{L:choice}. The independence of the choice of $\phi$ then 
implies that 
\[
I_*(Y)=I_*(X_{\phi'})=I_*(X_\phi)=\I_*(\Sigma,L).
\]  

To calculate $I_*(Y)$, we will use the Floer exact triangle of 
\cite{floer2}, see also \cite{BD}. Let $\gamma$ be a knot in an integral 
homology sphere $M$. Denote by $M - \gamma$ the integral 
homology sphere obtained by $(-1)$--surgery along $\gamma$, and by $Y = 
M + 0 \cdot \gamma$ the homology $S^1\times S^2$ obtained by 
$0$--surgery along $\gamma$. The instanton Floer homology groups of the three 
manifolds are then related by the Floer exact triangle of total degree 
$-1$\,:

\medskip

\begin{picture}(250,84)
        \put(170,64)          {$I_*\,(Y)$}
        \put(152,50)          {\vector(-1,-1){20}}
        \put(229,30)          {\vector(-1,1){20}}
        \put(110,12)          {$I_*\,(M)$}
        \put(160,15)          {\vector(1,0){50}}
        \put(230,12)          {$I_*\,(M - \gamma)$}
\end{picture}

\begin{example}
Let $L_n \subset S^3$ be the Hopf link with linking number $n$ as in 
Figure \ref{figureA} (where $n = 2$). Then $\I_*(S^3,L) = I_*(Y)$ for the 
manifold $Y$ as in Figure \ref{figureB}.

\begin{figure}[!ht]
 \begin{minipage}[b]{0.48\linewidth}
  \centering
  \psfrag{lk}{$\lk=n$}
  \includegraphics[scale=.75]{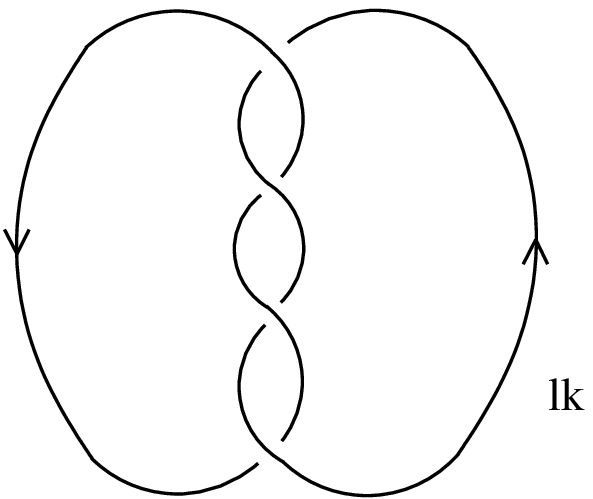}
  \caption{}\label{figureA}
 \end{minipage}
 \begin{minipage}[b]{0.48\linewidth}
  \centering
  \psfrag{0}{$0$}
  \psfrag{-1}{$-1$}
  \includegraphics[scale=.7]{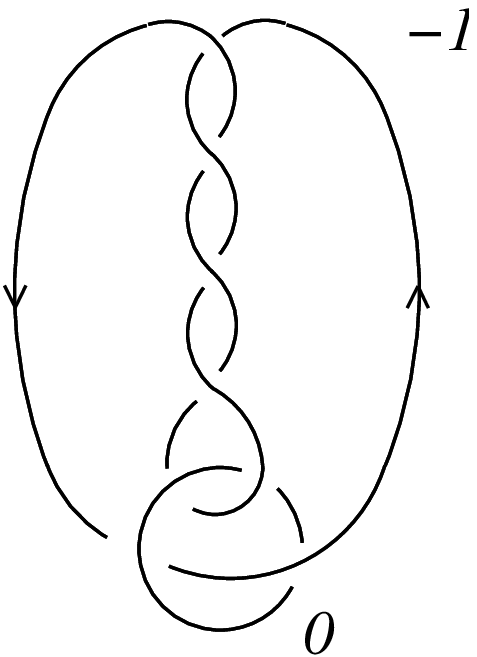}
  \caption{}\label{figureB}
 \end{minipage}
\end{figure}

\noindent
Apply the Floer exact triangle with $M = S^3$ viewed as the manifold obtained
by $(-1)$--surgery on the trivial knot framed by $-1$ in Figure \ref{figureB}.
Let $\gamma$ be the zero framed circle in Figure \ref{figureB}. Since $I_* 
(M) = 0$, we have an isomorphism $I_* (Y) = I_*(M -\gamma)$, where the 
manifold $M -\gamma$ has surgery description as shown in Figure 
\ref{figureC}. 

\begin{figure}[!ht]
\centering
\psfrag{blow down}{blow down}
\psfrag{-1}{$-1$}
 \includegraphics[scale=.75]{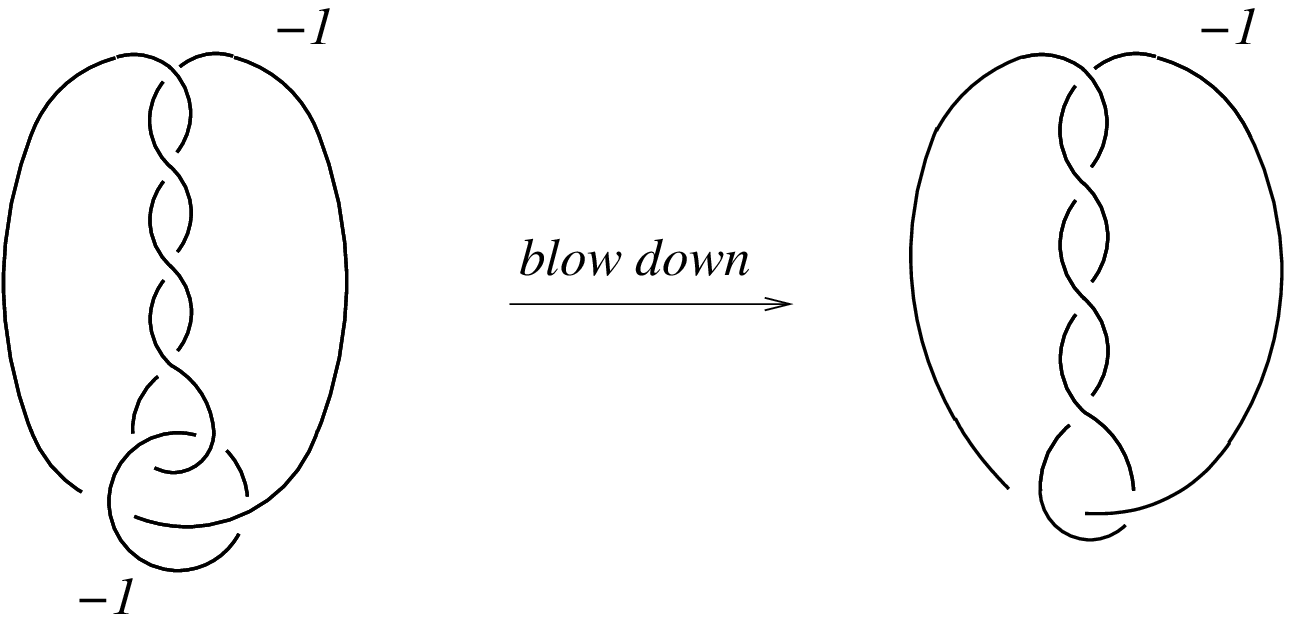}
\caption{}
\label{figureC}
\end{figure}

After the blow down, we see that $M - \gamma$ is the result of 
$(-1)$--surgery on a twist knot, hence is diffeomorphic to the Brieskorn 
homology sphere $\Sigma(2,3,6n+1)$ with reversed orientation; see for 
instance \cite[Figure 3.19]{Sav1}. Therefore (cf. Fintushel--Stern \cite{FS}) 
\begin{gather}
\I_*(S^3,L_n) = (\Z^{n/2},0,\,\Z^{n/2},0,\,\Z^{n/2},0,\,\Z^{n/2},0)
\;\;\text{if $n$ is even, and} \notag \\
\I_*(S^3,L_n) = (\Z^{(n-1)/2},0,\,\Z^{(n+1)/2},0,\,\Z^{(n-1)/2},0,\,\Z^{(n+1)/2},0)
\;\;\text{if $n$ is odd}.\notag
\end{gather}

\smallskip

\end{example}


\section{Relation with the instanton Floer homology of knots}
In this section, we will show that the Floer homology $\I_* (\Sigma,L)$ can
be viewed as a reduced version of the instanton knot Floer homology $KHI$ 
of Kronheimer and Mrowka.

The Floer homology $KHI$ is defined as follows; see for instance \cite{KM1}. 
Let $L = \ell_1\,\cup\ldots \cup\,\ell_r$ be an oriented link in a homology 
sphere $\Sigma$, and $F_r$ a genus-one surface with $r$ boundary components 
$\delta_1,\ldots,\delta_r$. Form a closed manifold $Z$ by attaching manifolds
$F_r \times S^1$ and $\Sigma - \Int N(L)$ to each other along their 
boundaries in such a fashion that each $\delta_i$ matches the canonical 
longitude of the component $\ell_i$ of $L$, and the $S^1$ factor matches the 
meridians. 

Let $E$ be a $U(2)$--bundle over $Z$ with $w = c_1 (E)$ dual to a curve $\nu 
\subset F_r$ representing a generator of the first homology of the closed 
genus-one surface obtained by closing up the $r$ boundary components of 
$F_r$ by discs. The instanton Floer homology with complex coefficients 
arising from the space of $PU(2)$--connections in $\ad (E)$ modulo the group 
of automorphisms of $E$ with determinant one will be denoted $I_*(Z)_w$. It 
has a relative grading by $\mathbb Z/8$. 

If $z \in Z$ is a point, the $\mu$--map gives us a degree-four operator 
$\mu(z): I_* (Z)_w \to I_* (Z)_w$ with eigenvalues $\pm 2$. The space $I_*
(Z)_w$ is then a sum of two subspaces of equal dimensions,
\[
I_* (Z)_w\; =\; I_* (Z)_{w,2}\,\oplus\, I_* (Z)_{w,-2},
\]
namely, the eigenspaces of $\mu(z)$ corresponding to the eigenvalues $2$ 
and $-2$. By definition,
\[
KHI (\Sigma,L) = I_* (Z)_{w,2}.
\]

Similarly, given a two-component link $L \subset \Sigma$, the Floer homology 
$\I_* (\Sigma,L)$ splits into a sum of the eigenspaces of the degree-four
operator $\I_*(\Sigma,L) \to \I_*(\Sigma,L)$ corresponding to the eigenvalues 
$2$ and $-2$. These eigenspaces have equal dimension. The $(+2)$--eigenspace 
will be denoted by $\I'_* (\Sigma,L)$.

\begin{theorem}\label{thm2}
Let $L = \ell_1\,\cup\,\ell_2$ be an oriented two-component link in a homology 
sphere $\Sigma$. Then 
\[
\chi (\I'_* (\Sigma,L))\; =\; \pm \lk (\ell_1,\ell_2)
\]
and 
\[
KHI (\Sigma,L)\; =\; \I'_* (\Sigma,L)\,\otimes\,H_*(T^2;\mathbb C).
\]
\end{theorem}

\begin{proof}
The first statement is a direct consequence of Theorem \ref{linking}. To 
prove the second statement, we will use a different description of the 
manifold $Z$. Remember that $Z$ is obtained by attaching $F_2 \times S^1$ 
to the exterior of the link $L$. We claim that $Z$ can be obtained by 
first attaching $F_1\times S^1$ to the knot $\ell_1\,\#\,\ell_2$, which is a 
band-sum of the knots $\ell_1$ and $\ell_2$, and then performing 0-surgery 
on a small loop $\gamma$ going once around the band with linking number zero. 

To prove this claim, we will resort to a 4-dimensional picture. Choose a 
compact oriented 4-manifold $W$ with boundary $\Sigma$. Then the manifold 
$Z$ will be the boundary of the 4-manifold obtained from $W$ by attaching 
$F_2 \times D^2$ to $\p W$ along the two solid tori $\p F_2\,\times D^2$. 
This operation can be done in two steps. First, we choose a properly 
embedded arc $J \subset F_2$ connecting the two boundary components of 
$F_2$. Let $N(J)$ with a tubular neighborhood of $J$ in $F_2$, and attach 
the 1-handle $N(J) \times D^2$ to $\p W$. What is left to attach is $F_1 
\times D^2$. It is attached along $\p F_1 \times D^2$ to the band-sum $\ell_1
\,\#\,\ell_2$, the band running geometrically once over the 1-handle. Since 
we are only interested in the boundary of the resulting 4-manifold, we trade 
the 1-handle for a 2-handle to complete the proof.

With the above description of the manifold $Z$ in place, the second statement 
of the theorem follows from Proposition 3.11 (2) of \cite[Part II]{BD}.
\end{proof}

Let $k_+$ and $k_-$ be knots in $\Sigma$ related by a single crossing change, 
and let $k_0$ be the corresponding two-component link, see Figure
\ref{figureD}. 

\bigskip

\begin{figure}[!ht]
\centering
\psfrag{k+}{$k_+$}
\psfrag{k-}{$k_-$}
\psfrag{k0}{$k_0$}
 \includegraphics[scale=.75]{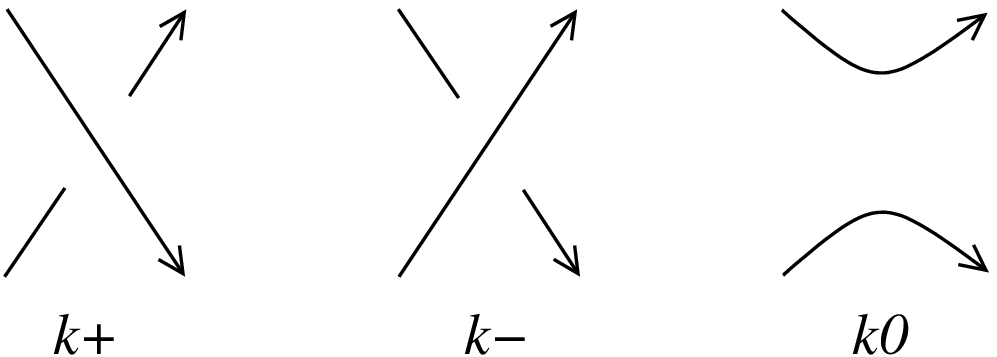}
\caption{}
\label{figureD}
\end{figure}

\noindent
The instanton Floer homology of these can be included into the following 
Floer exact triangle, see \cite[Theorem 3]{KM1},

\medskip

\begin{picture}(250,84)
        \put(150,64)          {$KHI\,(\Sigma,k_0)$}
        \put(152,50)          {\vector(-1,-1){20}}
        \put(229,30)          {\vector(-1,1){20}}
        \put(75,12)           {$KHI\,(\Sigma,k_+)$}
        \put(160,15)          {\vector(1,0){50}}
        \put(230,12)          {$KHI\,(\Sigma,k_-)$}
\end{picture}

\noindent
Observe that $\chi(KHI_*(k)) = 1$ for all knots $k \subset \Sigma$, see 
Kronheimer--Mrowka \cite[Theorem 1.1]{KM1} and also Braam--Donaldson 
\cite[Part II, Example 3.13]{BD}. This is consistent with the above exact 
triangle because
\[
\chi(KHI(\Sigma,k_0)) = \chi(\I'_*\,(\Sigma,k_0)\,\otimes\, H_*(T^2)) = 
\chi(\I'_*\,(\Sigma,k_0)) \cdot \chi (T^2) = 0
\]
by Theorem \ref{thm2}.


\end{document}